\documentclass[12pt]{amsart}

\usepackage{fullpage}   
\usepackage{amssymb,amsthm,amsmath,enumitem,url,verbatim, comment,xcolor}

\theoremstyle{plain}
\newtheorem{theorem}{Theorem}[section]
\newtheorem{lemma}[theorem]{Lemma}
\newtheorem{definition}[theorem]{Definition}
\newtheorem{proposition}[theorem]{Proposition}
\newtheorem{corollary}[theorem]{Corollary}
\newtheorem{conjecture}[theorem]{Conjecture}

\theoremstyle{remark}

\newtheorem*{acknowledgment}{Acknowledgment}

\numberwithin{equation}{section}

\newcommand{\seclabel}[1]{\label{sec:#1}}   

\usepackage{palatino, url, multicol} 

\makeatletter

\newcommand{\Rmnum}[1]{\expandafter\@slowromancap\romannumeral #1@}
\makeatother
\providecommand{\f}{\ensuremath{\frac{1}{2}}}

\newcommand{\Inn}{\mathrm{Inn}}
\newcommand{\Mlt}{\mathrm{Mlt}}
\newcommand{\Aut}{\mathrm{Aut}}

\title{Automorphic loops and metabelian groups}

\author[M. Greer]{Mark Greer}
\author[L. Raney]{Lee Raney}
\address{Department of Mathematics\\
One Harrison Plaza \\
University of North Alabama\\
Florence, AL 35632 USA }

\email{\url{mgreer@una.edu}}
\email{\url{lraney@una.edu}}


\begin{document}
\allowdisplaybreaks

\begin{abstract}
Given a uniquely 2-divisible group $G$, we study a commutative loop $(G,\circ)$ which arises as a result of a construction in \cite{baer}. We investigate some general properties and applications of $\circ$ and determine a necessary and sufficient condition on $G$ in order for $(G, \circ)$ to be Moufang. In \cite{greer14}, it is conjectured that $G$ is metabelian if and only if $(G, \circ)$ is an automorphic loop. We answer a portion of this conjecture in the affirmative: in particular, we show that if $G$ is a split metabelian group of odd order, then $(G, \circ)$ is automorphic.
\end{abstract}

\maketitle

\section{Introduction}
\seclabel{intro}
A loop $(Q,\cdot)$ consists of a set $Q$ with a binary operation $\cdot : Q\times Q\to Q$ such that (i) for all $a,b\in Q$, the equations $ax = b$ and $ya = b$ have unique solutions $x,y\in Q$, and (ii) there exists $1\in Q$ such that $1x = x1 = x$ for all $x\in Q$.  Standard references for loop theory are \cite{bruck71,pflugfelder90}.

Let $G$ be a uniquely $2$-divisible group, that is, a group in which the map $x\mapsto x^2$ is a bijection. On $G$ we define a new binary operations as follows:
\begin{align}
x\circ y &= xy[y,x]^{1/2}\,.   \label{eqn:groupgamma}
\end{align}
Here $a^{1/2}$ denotes the unique $b\in G$ satisfying $b^2 = a$ and $[y,x] = y^{-1}x^{-1}yx$. Though it is not obvious, $(G,\circ)$ is a commutative loop with neutral element $1$.  Moreover, this loop is \emph{power-associative}, which informally means that integer powers of elements can be defined unambiguously, and powers in $G$ and powers in $(G,\circ)$ coincide. It turns out that $(G,\circ)$ lives in a variety of loops called $\Gamma$-loops (defined in {\S}2 2), which include commutative RIF loops \cite{KV09} and commutative automorphic loops \cite{JKV11}.

If $G$ is nilpotent of class at most $2$, then $(G,\circ)$ is an abelian group. In this case, the passage from $G$ to $(G,\circ)$ is called the ``Baer trick'' \cite{isaacs}.  This construction seems to first appear in \cite{baer}. It was utilized by Bender in \cite{bender} to provide an alternative proof of the following result due to Thompson in \cite{thompson}.

\begin{theorem}
Let $p$ be an odd prime and let $A$ be the semidirect product of a $p$-subgroup $P$ with a normal $p'$-subgroup $Q$. Suppose that $A$ acts on a $p$-group $G$ such that
\[C_G(P) \leq C_G(Q).\]
Then $Q$ acts trivially on $G$.
\end{theorem}

Our goal is to study $(G,\circ)$ with different restrictions on $G$.  We show that $(G,\circ)$ is a commutative Moufang loop \emph{if and only if} $G$ is uniquely 2-divisible 2-Engel (Theorem \ref{moufang}) and give an alternative proof to Baer that if $(G,\circ)$ is an abelian group then $G$ has nilpotency class at most 2 
(Corollary \ref{newbaer}).  Our main result is that if $G$ is uniquely 2-divisible split-metabelian then $(G,\circ)$ is a commutative automorphic loop (Theorem \ref{splitmeta}).  Finally we end with some general facts about $(G,\circ)$ when $G$ is metabelian and open problems.

\section{Preliminaries}
\seclabel{background}
To avoid excessive parentheses, we use the following convention:
\begin{itemize}
\item multiplication $\cdot$ will be less binding than divisions $\backslash, /$.
\item divisions are less binding than juxtaposition
\end{itemize}
For example $xy/z \cdot y\backslash xy$ reads as $((xy)/z)(y\backslash (xy))$.  To avoid confusion when both $\cdot$ and  $\circ$ are in a calculation, we denote divisions by $\backslash_{\cdot}$ and $\backslash_{\circ}$ respectively.

In a loop \emph{Q}, the left and right translations by $x \in Q$ are defined by $yL_{x} = xy$ and $yR_{x} = yx$ respectively.  We thus have $\backslash,/$ as $x\backslash y=yL_{x}^{-1}$ and $y/x=yR_{x}^{-1}$.  We define the \emph{left section} of $Q$, $L_{Q}=\{L_{x}\mid x\in Q\}$, \emph{left multiplication group} of $Q$, $\Mlt_{\lambda}(Q)= \left\langle L_{x}\mid x\in Q\right\rangle$ and \emph{multiplication group} of $Q$, $\Mlt(Q)=\left\langle R_{x},L_{x}\mid x\in Q\right\rangle$.  We define the \emph{inner mapping group} of $Q$, $\Inn(Q)=\Mlt(Q)_{1}= \{\theta\in \Mlt(Q) \mid 1\theta=1\}$.  It is well known that $\Inn(Q)$ has the standard generators $L_{x,y}, R_{x,y},$ and $T_x$ (see \cite{bruck71}) where
\[L_{x,y} = L_xL_yL_{yx}^{-1}\qquad R_{x,y} = R_xR_yR_{xy}^{-1} \qquad T_x = R_xL_x^{-1}.\]

A loop \emph{Q} is an \emph{automorphic loop} if every inner mapping of $Q$ is an automorphism of $Q$, $\Inn(Q) \leq \Aut(Q)$.  A loop is Moufang if it satisfies $xy\cdot zx = x(yz\cdot x)$ and is a Bruck loop if it satisfies both $x(y \cdot xz) = (x\cdot yx)z$ and $(xy)^{-1} = x^{-1}y^{-1}$ where $x^{-1}$ is the unique two-sided inverse of $x$.

\begin{definition}
A loop $(Q,\cdot)$ is a $\Gamma$-loop if the following hold
\begin{itemize}
\item[($\Gamma_{1}$)]\quad $Q$ is commutative.
\item[($\Gamma_{2}$)]\quad $Q$ has the automorphic inverse property (AIP):\quad $\forall x,y\in Q$, $(xy)^{-1}=x^{-1}y^{-1}$.
\item[($\Gamma_{3}$)]\quad $\forall x\in Q$, $L_{x}L_{x^{-1}}=L_{x^{-1}}L_{x}$.
\item[($\Gamma_{4}$)]\quad $\forall x,y\in Q$, $P_{x}P_{y}P_{x}=P_{yP_{x}}$ where $P_{x}=R_{x}L_{x^{-1}}^{-1}=L_{x}L_{x^{-1}}^{-1}$.
\end{itemize}
\end{definition}

We recall some definitions and notation, which is standard in most group theory books.  We define $[x_{0},x_{1},\ldots,x_{n}]=[[[x_{0},x_{1}],\ldots],x_{n}]$.  Hence, $[x,y,z]=[[x,y],z]$.  The following identities are well-known:
\begin{lemma}
Let $x,y,z\in G$ for a group $G$.  
\begin{itemize}
\item  $[xy,z]=[x,z]^{y}[y,z]=[x,z][x,z,y][y,z]$
\item  $[x,yz]=[x,z][x,y]^{z}=[x,z][x,y][x,y,z]$
\item  $[x,y^{-1}]=[y,x]^{y^{-1}}$ and $[x^{-1},y]=[y,x]^{x^{-1}}$
\item  $[x,y^{-1},z]^{y}[y,z^{-1},x]^{z}[z,x^{-1},y]^{x}=[x,y,z^{x}][z,x,y^{z}][y,z,x^{y}]=1$
\end{itemize}
\end{lemma}
Recall that the \emph{lower central series} of a group is $G=\gamma_{1}(G)\geq \gamma_{2}(G) \geq \ldots$, with $\gamma_{i}(G)$ defined inductively by 
\begin{equation*}
\gamma_{1}(G)=G \qquad \gamma_{i+1}(G)=[\gamma_{i}(G),G] 
\end{equation*}  
and the \emph{upper central series} of a group $G$ is $1=\zeta^{0}(G)\leq \zeta^{1}(G) \leq \ldots$, with $\zeta^{i}(G)$ defined inductively by
\begin{equation*}
\zeta^{0}(G)=1 \qquad \zeta^{i+1}(G)/\zeta^{i}(G)=Z(G/\zeta^{i}(G))
\end{equation*} 
where if $\pi_{i}:G\rightarrow \zeta^{i}(G)$ is the natural projection map, then $\zeta^{i+1}(G)$ is the inverse image of the center.  

Finally, a group $G$ is \emph{nilpotent} if its upper central series has finite length $\Leftrightarrow$ its lower central series has finite length.  Therefore, we have $G$ is \emph{nilpotency of class n} $\Leftrightarrow [x_{0},x_{1}\ldots,x_{n}]=1 \ \forall x_{i}\in G$.  A group $G$ is $2$-\emph{Engel} if $[x,y,y] = 1$, alternatively $xx^y = x^yx$, for all $x,y\in G$.  Lastly recall the \emph{derived subgroup} of $G$, $G' =\langle[x,y] | x,y\in G\rangle$.  A group is \emph{metabelian} if $G''=1$ (or $[x,y][u,v]=[u,v][x,y]$ for all $x,y,u,v\in G$).

\begin{theorem}(\cite{baer})
Let $G$ be a uniquely $2-$divisible group.  For all $x,y\in G$, define $x\circ y=xy[y,x]^{\f}$.  Then $(G,\circ)$ is an abelian group if and only if $G$ is has nilpotency class 2.  Moreover, powers in $G$ coincide with powers in $(G,\circ)$.
\end{theorem}
Note that in the proof of the above theorem the restriction to class 2 only appears in the proof of associativity.  An immediate question is what properties does $(G,\circ)$ have without the restriction that $G$ be nilpotent of class 2?

\begin{theorem}(\cite{greer14})
Let $G$ be a uniquely $2$-divisible group.  Then $(G,\circ)$ is a $\Gamma$-loop.  Moreover, powers coincide in $G$ and $(G,\circ)$.
\label{newloop}
\end{theorem}
The main goal of \cite{greer14} was to establish a connection to Bruck loops and $\Gamma$-loops of odd order.
\begin{theorem}(\cite{greer14})
There is a one-to-one correspondence between left Bruck loops of odd order n and $\Gamma$-loops of odd order $n$.  That is

\begin{enumerate}
\item[(i)] If $(Q,\cdot)$ is a left Bruck loop of odd order $n$ with $1\in Q$ identity element, then $(Q,\circ)$ is a $\Gamma$-loop of order $n$ where $x\circ y = (1)L_xL_y[L_y,L_x]^{1/2}$.
\item[(ii)] If $(Q,\cdot)$ is a $\Gamma$-loop of odd order $n$, then $(Q,\oplus)$ is a left Bruck loop of order $n$ where $x\oplus y = (x^{-1}\backslash (y^2x))^{1/2}$.
\item[(iii)] The mappings in $(i)$ and $(ii)$ are mutual inverses.
\end{enumerate}
\end{theorem}
In general, not much can be said about $(G,\circ)$ without any restrictions on $G$.  However, we do have the following.
\begin{lemma}
Let $G$ be a uniquely $2-$divisible group.  Then $Z(G)\leq Z(G,\circ)$.  
\end{lemma}
\begin{proof}
Let $g\in Z(G)$.  Then we have

\[
g\circ (x \circ y)= gxy[y,x]^{\f}[xy[y,x]^{\f},g]^{\f}
=gxy[y,x]^{\f}
=gxy[y,gx]^{\f}
=(g\circ x)\circ y,
\]
\[
x\circ(g\circ y) =xgy[gy,x]^{\f}
=xgy[y,x]^{\f}
=xgy[y,xg]^{\f}
=(x\circ g)\circ y,
\]
\[
x\circ (y\circ g) =xyg[yg,x]^{\f}
=xyg[y,x]^{\f}
=xy[y,x]^{\f}g
=xy[y,x]^{\f}g[g,xy[y,x]^{\f}]^{\f}
=(x\circ y)\circ g.
\]
Thus $g\in Z(G,\circ)$.
\end{proof}

It turns out that $(G,\circ)$ has a lot of structure if $G$ is 2-Engel.
\begin{lemma}
Let $G$ be uniquely $2-$divisible.  Then $xy[y,x]^{1/2}=(xy^{2}x)^{1/2}$ if and only if $G$ is $2-$Engel.
\end{lemma}
\begin{proof} Before beginning the proof, we first note that if $G$ is uniquely $2-$divisible and $a, b \in G$ commute, then $a$ commutes with $b^{1/2}$. Indeed, since $(a^{-1} b^{1/2} a)^2 = a^{-1} b a$, it follows that $(a^{-1} b a)^{1/2} = a^{-1} b^{1/2} a$. Thus, since $a$ and $b$ commute, we have that $b^{1/2} = a^{-1} b^{1/2} a$, as desired.\\
Suppose $G$ is $2-Engel$. Hence, both $x$ and $y$ commute with $[y,x]$.  Then by the note above, \[(xy[y,x]^{1/2})^{2}=xy[y,x]^{1/2}xy[y,x]^{1/2}=(xy)^{2}[y,x]=xy^{2}x.\]
Taking square roots of both sides gives the desired results.\\
For the reverse direction, set $u=[y,x]^{1/2}$.  By hypothesis, $xyuxyu=xyyx$ and cancelling gives $uxyu=yx$.  Multiplying both sides on the right by $u$ gives $yxu=uxyu^2=uxyy^{-1}x^{-1}yx = uyx$.  Since $yx$ commutes with $u$ (Theorem \ref{newloop}) it commutes with any power of $u$.  Thus $yx[y,x]=[y,x]yx$.  Replacing $x$ with $y^{-1}x$ to get $x[y,y^{-1}x]=[y,y^{-1}x]x$.  But $[y,y^{-1}x]=y^{-1}x^{-1}yyy^{-1}x=[y,x]$.  Therefore $x[y,x]=[y,x]x$, that is, $[y,x,x]=1$.  Thus, $G$ is $2-$Engel.
\end{proof}

Defining multiplication with $x\oplus y=(xy^{2}x)^{\f}$ has been well studied by Bruck, Glaubermann, and others.

\begin{theorem}(\cite{glauberman64})
Let $G$ be uniquely $2-$divisible group.  For all $x,y\in G$, define $x\oplus y=(xy^2x)^{\f}$.  Then $(G,\oplus)$ is a Bruck loop.  Moreover, powers in $G$ coincide with powers in $(G,\circ)$.
\end{theorem} 
Finally, it is well known that commutative Bruck loops are Moufang \cite{bruck71}.
\begin{theorem}
Let $G$ be uniquely $2-$divisible.  Then $G$ is $2-$Engel if and only if $(G,\circ)$ is a commutative Moufang loop.
\label{moufang}
\end{theorem}
\begin{proof}
If $G$ is $2-$Engel then $(G,\circ)=(G,\oplus)$, and hence a commutative Bruck loop, so Moufang.\\  
Alternatively, set $u=[x,y]^{1/2}$.  Using the inverse property, \[y=x^{-1}\circ (x\circ y) = x^{-1}xyu^{-1}[xyu^{-1},x^{-1}]^{1/2}.\]  Cancel and multiply on the left by $u$ to get $u=[xyu^{-1},x^{-1}]^{1/2}$.  Squaring both sides gives $u^2 = [xyu^{1},x^{-1}] = [yu^{-1},x^{-1}]=uy^{-1}xyu^{-1}x^{-1}$.  Hence $u=y^{-1}xyu^{-1}xy$ after canceling.  Multiplying on the left by $x^{-1}$ to get $x^{-1}u=[x,y]u^{-1}x^{-1} = u^2u^{-1}x^{-1} = ux^{-1}$.  Since $x^{-1}$ commutes with $u$ it commutes with $u^2=[x,y]$.  Similarly, since $[x,y]$ commutes with $x^{-1}$, it commutes with $x$.  Hence, $G$ is 2-Engel.
\end{proof}

\section{Split Metabelian Groups}

Let $G$ be the semidirect product of a normal abelian subgroup $H$ of odd order acted on (as a group of automorphisms) by an abelian group $F$ of odd order. Products in $H$ and in $F$ are written multiplicatively. We use exponential notation for the action of $\mathrm{Aut}(H)$ on $H$: given $\theta \in \mathrm{Aut}(H)$, $h \in H$, define $h^\theta = \theta(h)$. 

Further, given $m, n \in \mathbb{Z}$ with $m$ and $n$ relatively prime to $|H|$, we make special use of the notation $h^{\frac{m}{n} \theta} = (h^\frac{m}{n})^\theta = (h^\theta)^{\frac{m}{n}}$. Note that since $H$ is abelian, this convention is consistent with an additional notation: given commuting automorphisms $\theta, \psi \in \mathrm{Aut}(H)$, $h^{\theta + \psi} = h^\theta h^\psi$. Then $G = H \rtimes F = HF$, where 
\[h_1f_1 h_2f_2 = h_1f_1 \cdot h_2f_2 = h_1h_2^{f_1} f_1 f_2\] 
for all $h_1, h_2 \in H, f_1, f_2 \in F$. Note that $G$ is metabelian (we refer to such groups as \emph{split metabelian}). To proceed, we need a proposition.

\begin{proposition}
\label{commutingauts}
Let $H$ be an abelian group of odd order. Suppose $\alpha$ and $\beta$ are commuting automorphisms of $H$ with odd order in $\textrm{Aut}(H)$. Then the map $h \mapsto h^{\alpha + \beta}$ is an automorphism of $H$.
\end{proposition}

\begin{proof}
Define $\phi:H \to H$ by $\phi(h) = h^{\alpha + \beta}$. Clearly, $\phi$ is a homomorphism. We will show that $\phi$ is injective. Suppose $h_0 \in H$ such that $\phi(h_0) = 1$. It follows that $h_0^\alpha = h_0^{-\beta}$, and thus
\[h_0^{\alpha^2} = (h_0^\alpha)^\alpha = (h_0^{-\beta})^\alpha = (h_0^\alpha)^{-\beta} = (h_0^{-\beta})^{-\beta} = h_0^{\beta^2}.\]
Now, since $\alpha, \beta$ are commuting, odd-ordered automorphisms of $H$, there exists some positive, odd integer $k$ such that $\alpha^k = id_H = \beta^k$. In particular,
\begin{align*}
h_0^{\alpha^k} &= h_0^{\beta^k}\\
(h_0^{\alpha^2})^{\alpha^{k-2}} &= (h_0^{\beta^2})^{\beta^{k-2}}\\
(h_0^{\beta^2})^{\alpha^{k-2}} &= (h_0^{\beta^2})^{\beta^{k-2}}\\
(h_0^{\alpha^{k-2}})^{\beta^2} &= (h_0^{\beta^{k-2}})^{\beta^2}.\\
\end{align*}
Since $\beta^2 \in \textrm{Aut}(H)$, it follows that $h_0^{\alpha^{k-2}} = h_0^{\beta^{k-2}}$. Continuing in this manner, we have that $h_0^\alpha = h_0^\beta$, and hence $h_0^\beta = h_0^{-\beta}$. Since $|H|$ is odd, this implies that $h_0 = 1$. Therefore, $\phi$ is an injective homomorphism $H \to H$ and is thus an automorphism of $H$.
\end{proof}

Since $F$ is abelian, Proposition \ref{commutingauts} implies that if $\theta$ is a $\mathbb{Q}$-linear combination of elements of $F$ (where the numerators and denominators of the coefficients are relatively prime to $|H|$), the map $H \to H$, $h \mapsto h^\theta$ is an automorphism of $H$ which commutes with any other such linear combination $\psi$. In particular, note that the aforementioned automorphism has an inverse in $\mathrm{Aut}(H)$; we denote this inverse by $h \mapsto h^{\theta^{-1}}$, and this map also commutes with $\psi$. We will use this fact throughout the following calculations without specific reference.

\begin{lemma}
\label{circprops}
Let $u = hf, x = h_1f_1, y = h_2f_2 \in G$. Then
\begin{itemize}
\item $u^{-1} = h^{-f^{-1}} f^{-1}$
\item $u^{1/2} = h^{(1+f^{1/2})^{-1}} f^{1/2}$
\item $[x, y] = h_1^{f_1^{-1}(-1 + f_2^{-1})} h_2^{f_2^{-1}(-f_1^{-1} + 1)} \in H$
\item $x \circ y = h_1^{\frac{1+f_2}{2}} h_2^{\frac{1+f_1}{2}} f_1f_2$
\item $x \backslash y = x \backslash_\circ y =  \left( h_1^{-1 - f_1^{-1}f_2} h_2^2 \right)^{(1+f_1)^{-1}} f_1^{-1}f_2$
\item $uL_{x,y} = \left( h^{(1+f_1)(1+f_2)} h_2^{1+ff_1 - f - f_1} \right)^{\frac{(1 + f_1f_2)^{-1}}{2}} f$
\end{itemize}
\end{lemma}

\begin{proof}
First, we compute
\begin{align*}
u \cdot h^{-f^{-1}} f^{-1} &= hf \cdot h^{-f^{-1}} f^{-1}\\
&= h h^{-f^{-1}f} ff^{-1}\\
&= hh^{-1}ff^{-1}\\
&= 1,
\end{align*}
and first item is proved.

Next, we compute
\begin{align*}
\left(h^{(1+f^{1/2})^{-1}} f^{1/2}\right)^2 &= h^{(1+f^{1/2})^{-1}} f^{1/2} \cdot h^{(1+f^{1/2})^{-1}} f^{1/2}\\
&= h^{(1+f^{1/2})^{-1}} h^{(1+f^{1/2})^{-1}f^{1/2}} f^{1/2}f^{1/2}.
\end{align*}
Setting $k = h^{(1+f^{1/2})^{-1}} \in H$ gives
\begin{align*}
\left(h^{(1+f^{1/2})^{-1}} f^{1/2}\right)^2 &= kk^{f^{1/2}} f\\
&= k^{1 + f^{1/2}} f\\
&= hf\\
&= u,
\end{align*}
and thus $u^{1/2} = h^{(1+f^{1/2})^{-1}} f^{1/2}$.

Now,
\begin{align*}
[x,y] &= x^{-1}y^{-1}xy\\
&= \left(h_1^{-f_1^{-1}}f_1^{-1} \cdot h_2^{-f_2^{-1}}f_2^{-1} \right) \left(h_1 f_1 \cdot h_2 \cdot f_2 \right)\\
&= \left(h_1^{-f_1^{-1}} h_2^{-f_2^{-1}f_1^{-1}} f_1^{-1}f_2^{-1} \right) \left(h_1 h_2^{f_1} f_1 f_2 \right)\\
&= h_1^{-f_1^{-1}} h_2^{-f_2^{-1}f_1^{-1}} \left(h_1 h_2^{f_1} \right)^{f_1^{-1} f_2^{-1}} f_1^{-1}f_2^{-1}f_1 f_2\\
&= h_1^{-f_1^{-1}+(f_1f_2)^{-1}} h_2^{-(f_1f_2)^{-1} + f_2^{-1}} \cdot 1\\
&= h_1^{f_1^{-1}(-1 + f_2^{-1})} h_2^{f_2^{-1}(-f_1^{-1} + 1)}.
\end{align*}

Next,
\begin{align*}
x \circ y &= h_1 f_1 \circ h_2 f_2\\
&= (h_1f_1)(h_2f_2) \cdot [h_2f_2, h_1f_1]^{1/2}\\
&= \left(h_1 h_2^{f_1} f_1 f_2 \right) \left(h_2^{f_2^{-1}(-1 + f_1^{-1})} h_1^{f_1^{-1}(-f_2^{-1} + 1)}\right)^{1/2}\\
&= h_1 h_2^{f_1} \left(h_2^{f_2^{-1}(-1 + f_1^{-1})} h_1^{f_1^{-1}(-f_2^{-1} + 1)}\right)^{\frac{f_1f_2}{2}} f_1 f_2\\
&= h_1^{1 + \frac{f_2 \left(-f_2^{-1} + 1 \right)}{2}} h_2^{f_1 + \frac{f_1 \left(-1 + f_1^{-1} \right)}{2}} f_1 f_2\\
&= h_1^{\frac{1+f_2}{2}} h_2^{\frac{1+f_1}{2}} f_1f_2.
\end{align*}

To compute $x \backslash y$, observe that
\begin{align*}
x \circ \left( h_1^{-1 - f_1^{-1}f_2} h_2^2 \right)^{(1+f_1)^{-1}} f_1^{-1}f_2 &= h_1 f_1 \circ \left( h_1^{-1 - f_1^{-1}f_2} h_2^2 \right)^{(1+f_1)^{-1}} f_1^{-1}f_2\\
&= h_1^{\frac{1 + f_1^{-1}f_2}{2}} \left( h_1^{-1 - f_1^{-1}f_2} h_2^2 \right)^{(1+f_1)^{-1} \left(\frac{1 + f_1}{2}\right)} f_1 f_1^{-1}f_2\\
&= h_1^{\frac{1+ f_1^{-1} f_2}{2} + \frac{-1 - f_1^{-1} f_2}{2}} h_2^{2/2} f_2\\
&= h_2 f_2\\
&= y,
\end{align*}
and thus $x \backslash y = \left( h_1^{-1 - f_1^{-1}f_2} h_2^2 \right)^{(1+f_1)^{-1}} f_1^{-1}f_2$.

Finally,
\footnotesize
\begin{align*}
uL_{x,y} &= (x \circ y) \backslash ((u \circ x) \circ y)\\
&= \left( h_1^{\frac{1+f_2}{2}} h_2^{\frac{1+f_1}{2}} f_1f_2 \right) \backslash \left(\left((h^{\frac{1+f_1}{2}} h_1^{\frac{1+f}{2}} ff_1\right) \circ h_2f_2\right)\\
&= \left( h_1^{\frac{1+f_2}{2}} h_2^{\frac{1+f_1}{2}} f_1f_2 \right) \backslash \left( \left(h^{\frac{1 + f_1}{2}} h_1^{\frac{1 + f}{2}}\right)^{\frac{1 + f_2}{2}} h_2^{\frac{1 + ff_1}{2}} ff_1 f_2 \right)\\
&= \left( h_1^{\frac{1+f_2}{2}} h_2^{\frac{1+f_1}{2}} f_1f_2 \right) \backslash \left( h^{\frac{(1+f_1)(1+f_2)}{4}} h_1^{\frac{(1+f)(1+f_2)}{4}} h_2^{\frac{1 + ff_1}{2}} ff_1f_2 \right)\\
&= \left( \left(h_1^{\frac{1+f_2}{2}} h_2^{\frac{1+f_1}{2}}\right)^{-1 - (f_1f_2)^{-1} (ff_1f_2)} \left(h^{\frac{(1+f_1)(1+f_2)}{4}} h_1^{\frac{(1+f)(1+f_2)}{4}} h_2^{\frac{1 + ff_1}{2}}\right)^2 \right)^{(1 + f_1f_2)^{-1}} (f_1f_2)^{-1}(ff_1f_2)\\
&= \left( \left(h_1^{\frac{1+f_2}{2}} h_2^{\frac{1+f_1}{2}}\right)^{-1 - f} \left(h^{\frac{(1+f_1)(1+f_2)}{2}} h_1^{\frac{(1+f)(1+f_2)}{2}} h_2^{1 + ff_1}\right) \right)^{(1 + f_1f_2)^{-1}} f\\
&= \left( h^{\frac{(1+f_1)(1+f_2)}{2}} h_1^{\frac{1 + f_2}{2} (-1 -f) + \frac{(1 + f)(1 + f_2)}{2}} h_2^{\frac{1 + f_1}{2} (-1 -f) + (1 + ff_1)} \right)^{(1 + f_1f_2)^{-1}} f\\
&= \left(h^{\frac{(1 + f_1)(1 + f_2)}{2}} h_1^0 h_2^{\frac{1 + ff_1 - f - f_1}{2}}\right)^{(1 + f_1f_2)^{-1}} f\\
&= \left( h^{(1+f_1)(1+f_2)} h_2^{1+ff_1 - f - f_1} \right)^{\frac{(1 + f_1f_2)^{-1}}{2}} f.
\end{align*}
\end{proof}

\begin{theorem}
Let $G$ be a split metabelian group of odd order. Then $(G, \circ)$ is an automorphic loop.
\label{splitmeta}
\end{theorem}

\begin{proof}
Since $(G, \circ)$ is commutative, for any $x, y \in G$, $L_{x,y} = R_{x,y}$ and $T_x = id_G$. Thus, to prove that $(G,\circ)$ is automorphic, it suffices to show that $L_{x,y}$ is a loop homomorphism. We must show that $uL_{x,y} \circ vL_{x,y} = (u \circ v)L_{x,y}$ for all $u,v,x,w \in G$. Thus, let $u = hf, v = kg, x = h_1f_1, y = h_2 f_2 \in G$. We first compute, by Lemma \ref{circprops}
\begin{align*}
&uL_{x,y} \circ vL_{x,y}\\
&= \left(\left( h^{(1+f_1)(1+f_2)} h_2^{1+ff_1 - f - f_1} \right)^{\frac{(1 + f_1f_2)^{-1}}{2}} f\right) \circ \left(\left( k^{(1+f_1)(1+f_2)} h_2^{1+gf_1 - g - f_1} \right)^{\frac{(1 + f_1f_2)^{-1}}{2}} g \right)\\
&= \left( \left( h^{(1+f_1)(1+f_2)} h_2^{1+ff_1 - f - f_1} \right)^{\frac{(1 + f_1f_2)^{-1}}{2}} \right)^{\frac{1 + g}{2}} \left( \left( k^{(1+f_1)(1+f_2)} h_2^{1+gf_1 - g - f_1} \right)^{\frac{(1 + f_1f_2)^{-1}}{2}}\right)^{\frac{1 + f}{2}} fg\\
&= \left( h^{\frac{(1+f_1)(1+f_2)(1+g)}{2}} k^{\frac{(1 + f_1)(1 + f_2)(1 +f)}{2}} h_2^{\frac{(1 + ff_1 - f - f_1)(1 + g)}{2} + \frac{(1 + gf_1 - g - f_1)(1 + f)}{2}}\right)^\frac{(1 + f_1f_2)^{-1}}{2} fg\\
&= \left( h^{\frac{(1+f_1)(1+f_2)(1+g)}{2}} k^{\frac{(1 + f_1)(1 + f_2)(1 +f)}{2}} h_2^{1 - fg + fgf_1 - f_1}\right)^\frac{(1 + f_1f_2)^{-1}}{2} fg.
\end{align*}

On the other hand,
\begin{align*}
&(u \circ v)L_{x,y}\\
&=\left(h^{\frac{1+g}{2}} k^{\frac{1+f}{2}} fg\right)L_{x,y}\\
&=\left( \left((h^{\frac{1+g}{2}} k^{\frac{1+f}{2}}\right)^{(1+f_1)(1+f_2)} h_2^{1+fgf_1 - fg - f_1} \right)^{\frac{(1 + f_1f_2)^{-1}}{2}} fg\\
&= \left( h^{\frac{(1 + g)(1+f_1)(1 + f_2)}{2}} k^{\frac{(1 + f)(1 + f_1)(1 + f_2)}{2}} h_2^{1+fgf_1 - fg - f_1} \right)^{\frac{(1 + f_1f_2)^{-1}}{2}} fg\\
&=uL_{x,y} \circ vL_{x,y}.
\end{align*}
\end{proof}

As an immediate corollary, we see that if $G$ is any group such that all groups of order $|G|$ are split metabelian, then $(G, \circ)$ is an automorphic loop. In particular, disregarding the cases where $G$ is abelian, we obtain the following.

\begin{corollary}
If $|G|$ is any one of the following (for distinct odd primes $p$ and $q$), then $(G,\circ)$ is automorphic.
\begin{itemize}
\item $pq$ (where $p$ divides $q - 1$)
\item $p^2 q$
\item $p^2 q^2$
\end{itemize}
\end{corollary}

\begin{corollary}
Let $p$ and $q$ be distinct odd primes with $p$ dividing $q - 1$. Then there is exactly one nonassociative, commutative, automorphic loop of order $pq$.
\end{corollary}
\begin{proof}
Let $G$ be a group of order $pq$.  Then $(G,\circ)$ is automorphic (Theorem \ref{splitmeta}).  Suppose $Q$ is a $\Gamma$-loop of order $pq$.  Then $(Q,\oplus)$ is a Bruck loop.  The only two options are (i) $(Q,\oplus)$ is abelian or (ii) $(Q,\oplus)$ is the unique nonassociative Bruck loop of order $pq$ \cite{KNV17}.  For (i), $Q=(Q,\oplus)$ and hence an abelian group (so automorphic).  For (ii), $(G,\oplus_\circ)=(Q,\oplus)$ must be the same nonassociative Bruck loop, and hence, $Q=(G,\circ)$.
\end{proof}

The only known examples where $(G,\circ)$ is not automorphic occur when $G$ is not metabelian.

\begin{conjecture}
Let $G$ be a uniquely 2-divisible group.  Then $(G,\circ)$ is automorphic if and only if $G$ is metabelian.
\end{conjecture}

For a general metabelian group $G$, we have the following results.
\begin{lemma}
Let $G$ be a uniquely $2$-divisible, metabelian group.  Then for all $x,y,z\in G$.
\begin{itemize}
\item $[[x,y]^{\f},z]=[[x,y],z]^{\f}$
\item $[x,y,z][z,x,y][y,z,x]=1$.
\label{3com}
\end{itemize}
\end{lemma}
\begin{theorem}
Let $G$ be uniquely $2-$divisible and metabelian.  Then  $\zeta^{2}(G) \trianglelefteq Z(G,\circ)$
\end{theorem}
\begin{proof}
If $g\in \zeta^{2}(G)$, then it is clear that $gT_{x}=x$.  We show $gL_{x,y}=g$.  First, it is clear that  $[g,x,y]=1 \Leftrightarrow [x,g,y]=1 \Leftrightarrow [x,y,g]=1$.  Thus, we have $[g,x]y=y[g,x]$ and $[x,y]g=g[x,y]$.  Now, 
\begin{align*}
y\circ(x\circ g)&=yxg[g,x]^{\f}[xg,y]^{\f}[[x,g]^{\f},y]^{\f}\\
&=yxg[g,x]^{\f}[xg,y]^{\f}\\
&=yxg[g,x]^{\f}[x,y]^{\f}[g,y]^{\f}\\
&=yxg[x,y]^{\f}[g,y]^{\f}[g,x]^{\f}\\
&=yx[x,y]^{\f}g[g,yx]^{\f}\\
&=yx[x,y]^{\f}g[g,yx]^{\f}[g,[x,y]^{\f}]^{\f}\\
&=(y\circ x)\circ g
\end{align*}
Hence, $gL_{x,y}=g$.
\end{proof}

\begin{theorem}
Let $G$ be uniquely $2-$divisible and of nilpotency class $3$.  Then $Z(G,\circ)=\zeta^{2}(G)$.
\label{nil2}
\end{theorem}
\begin{proof}
By the previous theorem, we have $\zeta^{2}(G)\leq Z(G,\circ)$.  From Lemma \ref{3com}, we have $[y,x,z][z,y,x]=[y,[x,z]]$ by interchanging $x$ and $y$.  Thus, 
\begin{equation*}
[[y,x]^{\f},z]^{\f}[[z,y]^{\f},x]^{\f}=[y,[x,z]^{\f}]^{\f}. \qquad (*)
\end{equation*}
Let $g\in Z(G,\circ)$.  We show $[g,x,y]=1$ for all $x,y\in G$ and therefore, $g\in \zeta^{2}(G)$.  Since $g\in Z(G,\circ)$, we have $g\circ(x\circ y)=x\circ (y\circ g)$.  Hence, we have
\begin{align*}
gxy[y,x]^{\f}[xy,g]^{\f}[[y,x]^{\f},g]^{\f}&=xyg[g,y]^{\f}[yg,x]^{\f}[[g,y]^{\f},x]^{\f} & \Leftrightarrow \\
xyg[g,xy][y,x]^{\f}[xy,g]^{\f}[[y,x]^{\f},g]^{\f}&=xyg[g,y]^{\f}[yg,x]^{\f}[[g,y]^{\f},x]^{\f} & \Leftrightarrow \\
[g,xy][y,x]^{\f}[xy,g]^{\f}[[y,x]^{\f},g]^{\f}&=[g,y]^{\f}[yg,x]^{\f}[[g,y]^{\f},x]^{\f} & \Leftrightarrow \\
[g,xy]^{\f}[y,x]^{\f}[[y,x]^{\f},g]^{\f}&=[g,y]^{\f}[yg,x]^{\f}[[g,y]^{\f},x]^{\f} & \Leftrightarrow \\
[g,y]^{\f}[g,x]^{\f}[g,x,y]^{\f}[y,x]^{\f}[[y,x]^{\f},g]^{\f}&=[g,y]^{\f}[y,x]^{\f}[y,x,g]^{\f}[g,x]^{\f}[[g,y]^{\f},x]^{\f} & \Leftrightarrow \\
[g,x,y]^{\f}[[y,x]^{\f},g]^{\f}&=[y,x,g]^{\f}[[g,y]^{\f},x]^{\f} &\Leftrightarrow \\
[g,x,y]^{\f}&=[[y,x]^{\f},g]^{\f}[[g,y]^{\f},x]^{\f} &\Leftrightarrow \\
[g,x,y]^{\f}&=[y,[x,g]^{\f}]^{\f}\ (*) &\Leftrightarrow \\
[[g,x]^{\f},y]^{\f}[[g,x]^{\f},y]^{\f}[[x,g]^{\f},y]^{\f}&=1 &\Leftrightarrow \\
[[g,x]^{\f},y]^{\f}&=1 &\Leftrightarrow\\
[g,x,y]&=1.
\end{align*}
\end{proof}
\begin{corollary}
Let $G$ be uniquely $2-$divisible and of nilpotency class $3$.  Then $(G,\circ)$ is a commutative loop of nilpotency class $2$.
\end{corollary}
\begin{proof}
We have as sets, $G/\zeta^{2}(G)=(G,\circ)/Z(G,\circ)$ by Theorem \ref{nil2}.  Now, since $G/\zeta^{2}(G)$ is an abelian group, the two sets have the same operation and thus, $(G,\circ)/Z(G,\circ)$ is an abelian group.
\end{proof}

Finally, we give an alternative proof of Baer's result that if $(G,\circ)$ is an abelian group, then $G$ is of nilpotency class at most 2.
\begin{corollary}
Let $G$ be uniquely $2-$divisible.  If $(G,\circ)$ is an abelian group, then $G$ is of class at most $2$.
\label{newbaer}
\end{corollary}
\begin{proof}
Since $(G,\circ)$ is an abelian group, $(G,\circ)$ is a commutative Moufang loop. Thus, $G$ is $2-$Engel, which implies $G$ is of class at most $3$. Thus, by Theorem \ref{nil2}, $G = \zeta^2(G)$, and hence $G$ has class at most 2.

\end{proof}

\begin{acknowledgment}
Some investigations in this paper were assisted by the automated deduction tool, \textsc{Prover9}, and the finite model builder, \textsc{Mace4}, both developed by McCune \cite{mccune09}.  Similarly, all presented examples were found and verified using the GAP system \cite{GAP} together with the LOOPS package \cite{GAPNV}.  We also thank the anonymous referee for several comments that improved the manuscript. 

\end{acknowledgment}

\bibliographystyle{amsplain}

\begin{thebibliography}{100}


\bibitem{baer}
R. Baer, Engelsche Elemente Noetherscher Gruppen, \emph{Math. Ann.} \textbf{133} (1957), 256-76.

\bibitem{bender}
H. Bender, \"{U}ber endliche gr\"{o}{\ss}ten $p'$-Normalteiler in $p$-aufl\"{o}sbaren Gruppen, \emph{Archiv} \textbf{18} (1967), 15-16.

\bibitem{bruck71}
R.H. Bruck, \emph{A {S}urvey of {B}inary {S}ystems}, Springer-Verlag, Berlin, 1971.

\bibitem{GAP}
The GAP Group, \emph{Groups, {A}lgorithms, and {P}rogramming},
  http://www.gap-system.org (2008).

\bibitem{glauberman64}
G. Glauberman, On loops of odd order I, \emph{Journal of Algebra} \textbf{8} (1964), 374-396.
  
\bibitem{greer14}
M. Greer, A class of loops categorically isomorphic to Bruck loops of odd order, \emph{Comm. Algebra} 42 (2014), no. 8, 3682-3697.
  
  
\bibitem{isaacs}
I.M. Isaacs, \emph{Finite Group Theory}, Graduate Studies in Mathematics, 92. American Mathematical Society, Providence, RI (2008).

\bibitem{JKV11}
P. Jedli\v{c}ka, M.K. Kinyon, and P. Vojt\v{e}chovsk\'{y}, The structure of commutative automorphic loops, \emph{Trans. Amer, Math. Soc.} \textbf{363} (2011), 365-384.

\bibitem{KV09}
M.K. Kinyon and P. Vojt\v{e}chovsk\'{y}, Primary decompositions in varieties of commutative diassociateive loops, \emph{Communications in Algebra} \textbf{37} (2009), 1428-1444.

\bibitem{KNV17}
M.K. Kinyon, G.P. Nagy, and P. Vojt\v{e}chovsk\'{y}, Bol loops and Bruck loops of order $pq$, \emph{Journal of Algebra}, \textbf{473} (2017), 481-512.

\bibitem{mccune09}
W.W. McCune, \emph{Prover9, {M}ace4}, http://www.cs.unm.edu/\~{}mccune/prover9/ (2009).

\bibitem{GAPNV}
G.P. Nagy and P. Vojt\v{e}chovsk\'{y}, \emph{Loops: Computing with quasigroups and loops}, http://www.math.du.edu/loops (2008).

\bibitem{SV19}

I. Stuhl and P. Vojt\v{e}chovsk\'{y}, Enumeration of involutory latin quandles, Bruck loops and commutative automorphic loops of odd prime power order, \emph{Proceedings of the 4th Mile High Conference on Nonassociative Mathematics, Contemporary Mathematics} \emph{721} (2019), 261-276.

\bibitem{pflugfelder90}
H.O. Pflugfelder, \emph{Quasigroups and {L}oops: {I}ntroduction}, Sigma Series in Pure Math, Berlin, 1990.
  
\bibitem{thompson}
J.G. Thompson, Fixed Points of $p$-groups acting on $p$-groups, \emph{Math. Z.} \emph{80} (1964), 12-13.

\end{thebibliography}

\end{document}